\documentclass[leqno,11pt]{amsart}

\usepackage{graphics}
\usepackage{amssymb, amsmath}
\usepackage{version}
\usepackage{fancybox}
\usepackage{bm}
\usepackage{color}
\usepackage{mathrsfs}

\theoremstyle{theorem}
\newtheorem{theorem}{\bf Theorem}[section]
\newtheorem{lemma}[theorem]{\bf Lemma}

\newtheorem{corollary}[theorem]{\bf Corollary}

\theoremstyle{definition}
\newtheorem{definition}[theorem]{\bf Definition}

\theoremstyle{remark}
\newtheorem{remark}[theorem]{\bf Remark}
\makeatletter

\@addtoreset{equation}{section}

\renewcommand{\d}{\text{\rm d}}
\newcommand{\vep}{\varepsilon}

\newcommand{\e}{{\rm e}}

\newcommand{\R}{\mathbb{R}}
\newcommand{\N}{\mathbb{N}}
\newcommand{\Z}{\mathbb{Z}}
\newcommand{\vierint}{\int_{0}^1\int_0^T\int_{\square}\int_{\Omega}}
\newcommand{\wtts}{\overset{2,2}{\rightharpoonup}}
\newcommand{\dv}{{\rm{div}}}
\newcommand{\T}{\mathbb{T}}

\allowdisplaybreaks[4]

\begin{document}

\title[Space-time homogenization for porous medium equations]{Space-time homogenization problems \\ for porous medium equations \\ with nonnegative initial data}

\author{Goro Akagi}
\author{Tomoyuki Oka}

\thanks{G.A.~is supported by JSPS KAKENHI Grant Number JP21KK0044, JP21K18581, JP20H01812, JP18K18715 and JP20H00117, JP17H01095. T.O.~is supported by Division for Interdisciplinary Advanced Research and Education, Tohoku University and Grant-in-Aid for JSPS Fellows (No.~JP20J10143).}
\address[Goro Akagi]{Mathematical Institute and Graduate School of Science, Tohoku University, Aoba, Sendai 980-8578, Japan}
\email{goro.akagi@tohoku.ac.jp}
\address[Tomoyuki Oka]{Graduate School of Science, Tohoku University, Aoba, Sendai 980-8578, Japan}
\email{tomoyuki.oka.q3@tohoku.ac.jp}
\date{\today}
%
\maketitle
\begin{abstract}
This paper concerns a space-time homogenization limit of nonnegative weak solutions to porous medium equations. In particular, the so-called homogenized matrix will be characterized in terms of solutions to cell problems, which drastically vary in a scaling parameter $r > 0$. A similar problem has already been studied in~\cite{AO1}, where the growth of the power nonlinearity is strictly restricted due to some substantial obstacles. In the present paper, such obstacles will be overcome by developing local uniform estimates for the gradients of nonnegative weak solutions.
\end{abstract}

\section{Introduction}

In this paper, we shall consider a homogenization limit as $\vep \to 0_+$ of nonnegative weak solutions $u_\vep = u_\vep(x,t)$ for the following Cauchy-Dirichlet problem:
\begin{alignat}{4}
\partial_t u_\vep &= \mathrm{div} \Big( a(\tfrac x \vep, \tfrac t {\vep^r}) \nabla u_\vep^m \Big) \quad && \mbox{ in } \Omega \times (0, T),\label{eq:1.1}\\
 u_{\vep} &= 0 && \mbox{ on } \partial \Omega \times (0, T), \label{eq:1.2}\\
 u_{\vep} &= u_0 \geq 0 && \mbox{ in } \Omega \times \{0\},\label{eq:1.3}
\end{alignat}
where $\partial_t = \partial/\partial t$, $\Omega$ is a bounded domain of $\R^N$ with smooth boundary $\partial \Omega$, $T > 0$ and
\begin{equation}\label{H1}
1 < m < +\infty, \quad u_0 \in L^{m+1}(\Omega), \quad u_0 \geq 0.
\end{equation}
Moreover, $a = a(y,s)$ is a $1$-periodically oscillating symmetric coefficient satisfying uniform ellipticity, that is, $a = a(y,s): \R^N \times \R_+ \to \R^{N \times N}$ fulfills
\begin{align}\label{H2}
 {}^ta(y,s) = a(y,s), \quad a(y+k e_j,s+\ell)= a(y,s),
\end{align}
where $e_j$ stands for the $j$-th vector of the canonical basis of $\R^N$, for $j=1,2,\ldots,N$ and $k, \ell \in \Z$, and there exists a positive constant $\lambda$ such that
\begin{align}\label{H3}
 \lambda |\xi|^2 \leq a(y,s) \xi \cdot \xi \leq |\xi|^2 \quad \mbox{ for } \ \xi \in \R^N.
\end{align}
In addition, we assume that
\begin{align}\label{H4}
a\in W^{1,1}(\R_+;L^{\infty}(\R^N)),
\end{align}
under which existence and uniqueness of weak solutions for \eqref{eq:1.1}--\eqref{eq:1.3} have already been proved in~\cite[Definition 1.1 and Theorem 1.2]{AO1}. Moreover, the following fundamental result is obtained in~\cite[Theorem 1.3]{AO1}, which concerns the homogenization for nonlinear diffusion under more general settings including fast diffusion equations and sign-changing data.

\begin{theorem}[\cite{AO1}]\label{T:AO1}
Let $1<m<+\infty$ and $0< r<+\infty$. Suppose that
\eqref{H1}--\eqref{H4} are satisfied.
Let $(\vep_n)$ be a sequence in $(0,1)$ such that $\vep_n \to 0_+$ and let $u_{\vep_n}$ be the unique weak solution on $[0,T]$ to \eqref{eq:1.1}--\eqref{eq:1.3} with $\vep = \vep_n$. Then there exist a {\rm (}not relabeled{\rm )} subsequence of $(\vep_n)$ and functions
\begin{align*}
u&\in W^{1,2}(0,T;H^{-1}(\Omega)) \cap C_{\rm weak}([0,T];L^{m+1}(\Omega)),\\
z&\in L^{2}(\Omega\times (0,T) ;L^{2}(0,1;H^{1}_{\mathrm{per}}(\square)/\R)),
\end{align*}
where $\square$ denotes the unit open cube in $\R^N$ as a domain for the $y$ variable {\rm (}see Notation below{\rm )}, such that
\begin{alignat}{4}
u^m &\in L^2(0,T;H^1_0(\Omega)),\label{um:L2H10}\\
u_{\vep_n}^m &\to u^m \quad \text{ weakly in }\ L^2(0,T;H^1_0(\Omega)), \label{eq:Thm1.1-1}\\
u_{\vep_n} &\to u \quad \text{ strongly in }\ L^\rho(0,T;L^{m+1}(\Omega))\label{eq:Thm1.1-2}
\end{alignat}
for any $\rho \in [1,+\infty)$ and
\begin{align}
&a(\tfrac{x}{\vep_n},\tfrac{t}{\vep_n^r})\nabla u_{\vep_n}^m \label{eq:Thm1.1-3}\\
 &\quad\wtts a(y,s) \left(\nabla u^m +\nabla_y z\right)\ \mbox{ in } \  [L^2(\Omega \times  (0,T)  \times \square \times (0,1))]^N,\nonumber
\end{align}
where $\nabla_y$ stands for the gradient in $y$ and $\wtts$ denotes the notion of weak two-scale convergence and it will be recalled in \S \ref{Ss:wtts} below. Moreover, the limit $u$ solves the weak form of the homogenized equation, for a.e.~$t \in (0,T)$,
\begin{equation}\label{eq:homeq}
\left\{
\begin{aligned}
&\langle \partial_tu(t),\phi\rangle_{H^1_0(\Omega)}+\int_{\Omega} j_{\rm hom}(x,t)\cdot\nabla\phi(x)\, \d x  =
0
\ \mbox{ for } \, \phi\in H^1_0(\Omega),\\
&u(\cdot,0)=u_0 \ \mbox{ in } \Omega
\end{aligned}
 \right.
\end{equation}
with a homogenized diffusion flux $j_{\rm hom}\in[L^2(\Omega\times (0,T))]^{N}$ given by
\begin{equation}\label{eq:jhom}
j_{\rm hom}(x,t)=\int_0^1\int_{\square}a(y,s)\left[\nabla u^m(x,t)+\nabla_yz(x,t,y,s)\right]\, \d y\,\d s
\end{equation}
for a.e.~$t \in (0,T)$.
\end{theorem}

We further refer the reader to~\cite{Ji,NaRa}, where similar results are obtained (see also~\cite{Vis07}). As for degenerate $p$-Laplace parabolic equations, homogenization problems involving the scale parameter $r$ are discussed in \cite{EfPa, Wo}. The main result of the present paper reads,
\begin{theorem}\label{T:main}
Let $1<m<+\infty$ and $0< r<+\infty$. In addition to \eqref{H1}--\eqref{H4}, suppose that
\begin{equation}\label{hyp-u0}
 \log u_0 \in L^1_{\rm loc}(\Omega) \ \mbox{ if } \ m = 3\/{\rm ;} \quad u_0^{3-m} \in L^1_{\rm loc}(\Omega) \ \mbox{ if } \ m > 3.
\end{equation}
Let $u$ be a {\rm (}homogenized{\rm )} limit of the unique weak solutions $(u_{\vep_n})$ to \eqref{eq:1.1}--\eqref{eq:1.3} for a sequence $\vep_n \to 0_+$ such that \eqref{um:L2H10}--\eqref{eq:Thm1.1-3} are fulfilled, and hence, $u$ is a weak solution of the homogenized equation \eqref{eq:homeq}. Then the homogenized flux $j_{\rm hom}(x,t)$ defined by \eqref{eq:jhom} is rewritten as 
$$
j_{\rm hom} = a_{\rm hom} \nabla u
$$
for some $N \times N$ matrix $a_{\rm hom}$, which is constant if $r \neq 2$ but may depend on $(x,t)$ if $r = 2$. Moreover, $a_{\rm hom}$ can be characterized as follows\/{\rm :}

In case $0<r<2$, $a_{\rm hom}$ is a constant $N\times N$ matrix given by
\begin{align}\label{eq:ahom}
a_{\rm hom}e_k=\int_0^1\int_{\square}a(y,s) \left[\nabla_y\Phi_k+e_{k}\right]\, \d y\,\d s,
\end{align}
where $e_k \in \R^N$ denotes the vector with a $1$ in the $k$-th coordinate and $0$'s elsewhere and $\Phi_k\in L^2(0,1;H^1_{\mathrm{per}}(\square)/\R)$ is the unique weak solution to the cell-problem\/{\rm :}
\begin{equation}\label{eq:CP1}
-\dv_y \left(a(y,s) \left[ \nabla_y\Phi_k+e_{k}\right]\right)=0\ \text{ in }\ \T^N \times \T,
\end{equation}
for $k = 1,2,\ldots,N$. Here $\T^N$ and $\T$ stand for the $N$- and one-dimensional tori, respectively. Furthermore, the pair $(u,z)$ satisfying \eqref{um:L2H10}--\eqref{eq:jhom} is uniquely determined. Hence $(u_{\vep_n})$ converges to $u$ {\rm(}without taking any subsequence{\rm )}. Moreover, the function $z=z(x,t,y,s)$ can be written as
\begin{equation}\label{eq:z}
z(x,t,y,s) = \sum_{k=1}^N \left[ \partial_{x_k} u^m(x,t) \right] \Phi_k(y,s).
\end{equation}

In case $r=2$, $a_{\rm hom} = a_{\rm hom}(x,t)$ may depend on $(x,t)$ and is characterized by \eqref{eq:ahom} with $\Phi_k$ given by
\begin{align}\label{eq:ahom2}
\Phi_k(x,t,y,s)
= \begin{cases}
m u^{m-1}(x,t) \Psi_k(x,t,y,s) &\mbox{if } \ u(x,t) \neq 0,\\
0 &\mbox{if } \ u(x,t) = 0,
   \end{cases}
\end{align}
where $\Psi_k \in L^\infty(\Omega \times (0,T) ; L^2(0,1;H^1_{\rm per}(\square)/\R))$ is the unique weak solution to the cell problem,
\begin{equation}\label{eq:CP2}
\begin{cases}
\partial_s\Psi_k=\dv_y\left(a(y,s)\left[ mu^{m-1}(x,t)\nabla_y\Psi_k+e_{k}\right]\right) &\mbox{in } \, \T^N \times \T,\\
\Psi_k|_{s=0} = \Psi_k|_{s=1} &\mbox{in } \, \T^N
\end{cases}
\end{equation}
for each $(x,t) \in  [u\neq 0]:= \{(x,t) \in \Omega\times (0,T)\colon u(x,t) \neq 0\}$. Furthermore, $z$ is given as in \eqref{eq:z}, where $\Phi_k$ may depend on $(x,t)$ as well as $(y,s)$.

In case $2<r<+\infty$, $a_{\rm hom}$ is a constant $N\times N$ matrix given by
\begin{equation*}
a_{\rm hom}e_k=\int_{\square} \Big(\int_0^1a(y,s)\ \d s\Big) \left[\nabla_y\Phi_k+e_{k}\right]\, \d y,
\end{equation*}
where $\Phi_k\in H^1_{\mathrm{per}}(\square)/\R$ is the unique weak solution to the cell problem,
\begin{equation}\label{eq:CP3}
-\dv_y \left( \Big(\int_0^1a(y,s)\ \d s \Big) \left[\nabla_y\Phi_k+e_{k}\right] \right)=0\ \text{ in }\ \T^N.
\end{equation}
Furthermore, the pair $(u,z)$ satisfying \eqref{um:L2H10}--\eqref{eq:jhom} is uniquely determined. Hence $(u_{\vep_n})$ converges to $u$ {\rm(}without taking any subsequence{\rm )}. Finally, $z$ is independent of $s$ and given by \eqref{eq:z} with $\Phi_k = \Phi_k(y)$.
\end{theorem}

The theorem above provides a characterization of the homogenized matrix $a_{\rm hom}$ in terms of solutions to cell problems depending on the scale parameter $r$; in particular, one can observe that the case $r=2$ is critical. A similar result has already been established for $0 < m < 2$ in~\cite[Theorem 1.4]{AO1}, which however cannot cover the range $m \geq 2$ due to a substantial difficulty. A novelty of the present paper resides in the extension of the result in~\cite{AO1} to the range $m \geq 2$. In~\cite{AO1}, the proof relies on weak two-scale convergences of gradients for $u_\vep$ as well as $u_\vep^m$ as $\vep \to 0_+$, and then, they yield \emph{very weak} two-scale convergences of both $u_\vep$ and $u_\vep^m$ as well (see \S \ref{Ss:wtts} below for definitions), that is,
\begin{align*}
u_\vep(x,t) &= u(x,t) + \vep w(x,t,\tfrac{x}\vep,\tfrac{t}{\vep^r}) + o(\vep),\\
u_\vep^m(x,t) &=  u^m(x,t) + \vep z(x,t,\tfrac{x}\vep,\tfrac{t}{\vep^r}) + o(\vep) \ \mbox{ as } \ \vep \to 0_+
 \end{align*}
in a proper sense for some integrable functions $w = w(x,t,y,s)$ and $z = z(x,t,y,s)$ (see \S \ref{Ss:wtts}). Hence, similarly to formal asymptotic expansion (see \cite{BLP} for details), one can obtain (but rigorously here) the relation, 
$$
\partial_s w = \mathrm{div}_y \left( a(y,s) \left[\nabla_y z + \nabla u^m(x,t)\right] \right) \ \mbox{ in } \T^N \times \T,
$$
which will be used to derive cell problems. Moreover, a key ingredient of the proof is to verify the relation between $w$ and $z$, that is, $z = m u^{m-1} w$, which finally yields the representation of cell problems. Here we stress that both of course coincide for the linear case $m = 1$ (see~\cite{Ho}) and this procedure essentially arises from the nonlinear setting $m \neq 1$. Thus weak two-scale convergences (up to a subsequence) play a crucial role in the proof, and they follow in general from boundedness of sequences in Lebesgue spaces (see \S \ref{Ss:wtts} below). One can easily obtain uniform estimates for $\nabla u_\vep^m$ in $L^2(\Omega \times (0,T))$ due to an energy structure of the equation \eqref{eq:1.1}. On the other hand, it is more delicate to derive uniform estimates for $\nabla u_\vep$ in $L^2(\Omega \times (0,T))$, which are actually established in~\cite{AO1} only for $0 < m < 2$; indeed, one cannot generally expect the same estimates for $m \geq 2$ (see Remark \ref{R:opt} below for more details). In this paper, we shall develop \emph{local} (in space) uniform estimates for $\nabla u_\vep$ under certain assumptions, and then, they entail \emph{local} weak (resp., \emph{local} very weak) two-scale convergence for $\nabla u_\vep$ (resp., for $u_\vep$). Such a local weak two-scale convergence is actually weaker than those proved in~\cite{AO1}; however, we shall finally show that the local convergence is still enough to prove the main result mentioned above.

Moreover, repeating the same argument as in the proof of Theorem 1.6 in~\cite{AO1}, we can also verify the following:
\begin{corollary}\label{C:corr}
Let $1<m<+\infty$ and $0< r<+\infty$. Under the same assumptions as in Theorem \ref{T:main}, assume that 
\begin{itemize} 
 \item[\rm (H)] $a \in L^\infty(0,1;C^\alpha_{\rm per}(\square))$ for some $\alpha \in (0,1)$ if $r \neq 2$\,{\rm ;} $a$ is smooth in $(y,s)$ if $r=2$.
\end{itemize}
Let $u$ be a limit of weak solutions $(u_{\vep_n})$ to \eqref{eq:1.1}--\eqref{eq:1.3} as $\vep_n \to 0_+$ and let $\Phi_{k}$ be the solution of the cell-problem, that is, \eqref{eq:CP1} if $r \in (0,2)$\/{\rm ;} \eqref{eq:CP2} along with \eqref{eq:ahom2} associated with the limit $u$ if $r =2$\/{\rm ;} \eqref{eq:CP3} if $r \in (2,+\infty)$. Then it holds that
\begin{align*}
\lim_{\vep_n\to 0_+} \int_{0}^T\int_{\Omega} \Big| \nabla u_{\vep_n}^m-\nabla u^m -\sum_{k=1}^N (\partial_{x_k} u^m )\nabla_y\Phi_k (x,t,\tfrac{x}{\vep_n},\tfrac{t}{\vep_n^r} ) \Big|^2\, \d x\d t=0.
\end{align*}
Here $\Phi_k$ depends only on $(y,s)$ for $r \in (0,2)$ and on $y$ for $r \in (2,+\infty)$, respectively.
\end{corollary}

The rest of the present paper consists of three sections. In Section \ref{S:Pre}, we shall briefly review preliminary facts to be used for proving the main result. To be precise, \S \ref{Ss:wtts} is devoted to a brief summary of the two-scale convergence theory; moreover, \S \ref{Ss:bdd} provides a lemma on (standard) uniform estimates for weak solutions to \eqref{eq:1.1}--\eqref{eq:1.3}. In Section \ref{S:lue}, we shall derive \emph{local} uniform estimates for the gradient $\nabla u_\vep$. Based on them, in Section \ref{S:prf}, we shall give a proof for Theorem \ref{T:main}.

\bigskip
\noindent
{\bf Notation.} Throughout the present paper, $\nabla$ and $\nabla_y$ denote gradient operators with respect to $x$ and $y$, respectively. Moreover, we write $\d Z = \d x \,\d y \,\d t \,\d s$. For $A,B \subset \R^N$, we write $A \Subset B$, when the closure $\overline{A}$ of $A$ in $\R^N$ is included in $B$. Moreover, $\square = (0,1)^N$ is the unit cube in $\R^N$ and $\langle\cdot\rangle_y$ denotes the mean in $y\in \square$, that is,
$\langle g \rangle_y := \int_\square g(y) \, d y$ for $g \in L^1(\square)$.
Define $C^{\infty}_{\rm per}(\square)$ by 
$\{w\in C^{\infty}(\square) \colon w(\cdot+ k e_j)=w(\cdot) \text{ in } \R^N \text{ for } 1\leq j \leq N \text{ and } k\in \Z \}$, and then, 
we also define $W^{1,q}_{\rm per}(\square)$ and $L^q_{\rm per}(\square)$ as closed subspaces of $W^{1,q}(\square)$ and $L^q(\square)$ by
$
W^{1,q}_{\rm per}(\square) = \overline{C^\infty_{\rm per}(\square)}^{\|\cdot\|_{W^{1,q}(\square)}}$, $ L^q_{\rm per}(\square) = \overline{C^\infty_{\rm per}(\square)}^{\|\cdot\|_{L^q(\square)}}$, respectively, for $1\leq q < +\infty$. In particular, we set $H^1_{\rm per}(\square) := W^{1,2}_{\rm per}(\square)$ and write $L^q(\square)$ instead of $L^q_{\rm per}(\square)$, unless any confusion may arise.
Furthermore, we set $H^{1}_{\mathrm{per}}(\square)/\R :=\{w\in H^{1}_{\mathrm{per}}(\square)\colon \langle w\rangle_{y}=0\}$. 
Let $X$ be a normed space  with a norm $\|\cdot\|_X$ and a duality pairing $\langle \cdot, \cdot \rangle_X$ between $X$ and its dual space $X^*$  and denote by $C_{\rm weak}([0,T];X)$ the set of all weakly continuous functions defined on $[0,T]$ with values in $X$.  Moreover, we write $X^N = X \times X \times \cdots \times X$ ($N$-product space), e.g., $[L^2(\Omega)]^N = L^2(\Omega;\R^N)$.

\section{Preliminaries}\label{S:Pre}

\subsection{Two-scale convergence theory}\label{Ss:wtts}

In this subsection, we briefly review a \emph{space-time two-scale convergence theory} developed in~\cite{Ho}.
The notion of two-scale convergence was first proposed by G.~Nguetseng~\cite{Ng}, and then, developed by G.~Allaire~\cite{Al1} (see also, e.g.,~\cite{LNW,Vis06,Zhi}).
Throughout this section, we always assume that
\[
  0<r<+\infty,\quad 1< q< +\infty,
\]
unless noted otherwise.
The notion of \emph{weak space-time two-scale convergence} is defined by

\begin{definition}[Weak space-time two-scale convergence]
A bounded sequence $(v_{\vep})$ in $L^q(\Omega\times (0,T))$ is said to \emph{weakly space-time two-scale converge} to a function $v$ in $L^q(\Omega\times (0,T)\times\square\times (0,1))$ as $\vep\to 0_+$, if it holds that
\begin{align*}
 \lefteqn{\lim_{\vep\to 0_+}\int_0^T\int_{\Omega}v_{\vep}(x,t)\phi(x)b(\tfrac{x}{\vep})\psi(t)c(\tfrac{t}{\vep^r})\, \d x\d t}\\
 &= \vierint v(x,t,y,s)\phi(x)b(y)\psi(t)c(s)\, \d Z\nonumber
\end{align*}
for any $\phi\in C^{\infty}_{\rm c}(\Omega)$, $b\in C^{\infty}_{\mathrm{per}}(\square)$, $\psi\in C^{\infty}_{\rm c}(0,T)$ and $c\in C^{\infty}_{\rm per}([0,1])$.
Then we write
\[
   v_{\vep}\wtts v\quad \text{ in } L^q(\Omega\times\square\times (0,T)\times (0,1)).
\]
\end{definition}

The following theorem is concerned with weak two-scale compactness of bounded sequences in $L^q(\Omega\times (0,T))$.
\begin{theorem}[Weak space-time two-scale compactness]\label{T:wttscpt}
For any bounded sequence $(v_{\vep})$ in $L^q(\Omega\times (0,T))$, there exist a subsequence $\vep_n \to 0_+$ of $(\vep)$ and a function $v\in L^q(\Omega\times (0,T)\times \square\times (0,1)) $ such that
\[
  v_{\vep_n} \wtts v\quad \text{ in } L^q(\Omega\times (0,T)\times \square\times (0,1)) .
\]
\end{theorem}

\begin{proof}
See \cite[Theorem 2.3]{Ho}.
\end{proof}

Furthermore, as for the two-scale compactness of gradients, we have
\begin{theorem}[Weak space-time two-scale compactness for gradients]\label{T:wttscpt_grad}\
Let $(v_{\vep})$ be a bounded sequence in $L^{q}(0,T;W^{1,q}(\Omega))$ such that $v_{\vep_n} \to v$ in $L^q(\Omega \times (0,T))$ for some subsequence $(\vep_n)\to 0_+$ of $(\vep)$ and a limit $v$.
Then there exist a {\rm(}not relabeled{\rm )} subsequence of $\vep_n\to 0_+$ and a function $z\in L^q(\Omega\times (0,T);L^{q}(0,1; W^{1,q}_{\mathrm{per}}(\square)/\R))$ such that
\begin{align*}
  \nabla v_{\vep_n}\wtts \nabla v+\nabla_y z\quad \text{ in }\ [L^q(\Omega\times (0,T)\times \square\times (0,1))]^N.
\end{align*}
\end{theorem}

\begin{proof}
See \cite[Theorem 3.1]{Ho}. 
\end{proof}

As a corollary, we obtain
\begin{corollary}[Very weak two-scale convergence]\label{C:vw}
Under the same assumptions as in Theorem \ref{T:wttscpt_grad}, it holds that
\begin{align*}
\lim_{\vep_n\to 0_+} &\int_{0}^{T}\int_{\Omega}\frac{v_{\vep_n}(x,t)-v(x,t)}{\vep_n} \phi(x)b(\tfrac{x}{\vep_n})\psi(t)c(\tfrac{t}{\vep_n^{r}})\, \d x\d t\\
&= \vierint z(x,t,y,s) \phi(x)b(y)\psi(t)c(t)\, \d Z\nonumber
\end{align*}
for any $\phi\in C^{\infty}_{\rm c}(\Omega)$, $b\in C^{\infty}_{\mathrm{per}}(\square)/\R$ {\rm (}i.e., $\langle b \rangle_y = 0${\rm )}, $\psi\in C^{\infty}_{\rm c}(0,T)$ and $c\in C^{\infty}_{\mathrm{per}}([0,1])$.
\end{corollary}

\begin{proof}
See \cite[Corollary 3.3]{Ho}.
\end{proof}

\subsection{Uniform estimates}\label{Ss:bdd}

In this subsection, we shall recall uniform estimates established in~\cite[Lemma 4.1]{AO1} for solutions $u_\vep$ to \eqref{eq:1.1}--\eqref{eq:1.3} as $\vep \to 0_+$.

\begin{lemma}[Uniform estimates~\cite{AO1}]\label{L:bdd}
Let $0<m,r<+\infty$. For each $\vep \in (0,1)$ let $u_{\vep}\in W^{1,2}(0,T;H^{-1}(\Omega))$ be the unique weak solution to \eqref{eq:1.1}--\eqref{eq:1.3} under the same assumptions as in Theorem \ref{T:AO1}. Then the following {\rm(i)} and {\rm (ii)}  hold true\/{\rm :}
\begin{enumerate}
 \item[(i)] $(u_{\vep}^m) $ is bounded in $L^2(0,T;H^1_{0}(\Omega))\cap L^{\infty}(0,T;L^{(m+1)/m}(\Omega))$ and $(u_{\vep})$ is bounded in $L^{\infty}(0,T;L^{m+1}(\Omega))$,
 \item[(ii)] $(\partial_tu_{\vep})$ is bounded in $L^2(0,T;H^{-1}(\Omega))$
\end{enumerate}
for $\vep \in (0,1)$. In addition, if $m\in (1,2)$, 
then it holds that
\begin{enumerate}
 \item[(iii)] $(u_{\vep})$ is bounded in $ L^{2}(0,T;H^1_0(\Omega))$ for $\vep \in (0,1)$.
\end{enumerate}
\end{lemma}

\section{Local uniform estimates for $\nabla u_\vep$}\label{S:lue}

In this section, we shall establish \emph{local} (in space) uniform estimates for $\nabla u_\vep$, which will be used to derive \emph{very weak} two-scale convergence of $(\nabla u_\vep)$ as $\vep \to 0_+$, and then, it will play a crucial role in the proof of Theorem \ref{T:main}. Furthermore, we shall see in Remark \ref{R:opt} below that these local uniform estimates are intrinsic to the porous medium equation and they may not be substantially improved anymore. For simplicity, we shall first exhibit a formal derivation, which may not be justified immediately under the present setting (in particular, due to the regularity of weak solutions). However, we shall also mention how to justify such a formal derivation later. In what follows, we shall simply write $a_\vep$ instead of $a(\tfrac x \vep, \tfrac t {\vep^r})$, unless any confusion may arise. For each $\omega \Subset \Omega$, let $\rho \in C^\infty_{\rm c}(\Omega)$ be such that
$$
0 \leq \rho \leq 1 \ \mbox{ in } \Omega, \quad \rho \equiv 1 \ \mbox{ on } \omega, \quad \mathrm{supp} \, \rho \Subset \Omega
$$
and set $M_\rho := \sup_{x \in \Omega} |\nabla \rho(x)| < + \infty$.

We first handle the case $m \neq 2,3$. Formally, test \eqref{eq:1.1} by $(2-m)^{-1} u_\vep^{2-m} \rho^2$, which satisfies the homogeneous Dirichlet condition. Then we see that
\begin{align*}
  \frac1{(2-m)(3-m)} \frac{\d}{\d t} \left( \int_\Omega u_\vep^{3-m} \rho^2\, \d x \right) + \frac 1 {2-m}\int_\Omega a_\vep \nabla u_\vep^m \cdot \nabla \left( u_\vep^{2-m} \rho^2 \right)\, \d x= 0.
\end{align*}
Here we observe that
\begin{align*}
\lefteqn{
\frac 1 {2-m} \int_\Omega a_\vep \nabla u_\vep^m \cdot \nabla \left( u_\vep^{2-m} \rho^2 \right) \, \d x
}\\
&= m  \int_\Omega a_\vep \nabla u_\vep \cdot (\nabla u_\vep) \rho^2\, \d x
+ \frac{2m}{2-m} \int_\Omega u_\vep a_\vep \nabla u_\vep \cdot (\nabla \rho) \rho \, \d x\\
&\geq \lambda m \int_\Omega |\nabla u_\vep|^2 \rho^2 \, \d x - \frac{2m M_\rho}{m-2} \int_\Omega |u_\vep| |\nabla u_\vep| \rho \, \d x\\
&\geq \frac{\lambda m}2 \int_\Omega |\nabla u_\vep|^2 \rho^2 \, \d x - C_\omega \int_\Omega |u_\vep|^2 \, \d x.
\end{align*}
Hence using (i) of Lemma \ref{L:bdd}, we deduce that, for any $\omega \Subset \Omega$, there exists a constant $C_\omega \geq 0$ such that
\begin{equation}\label{ei1}
\frac1{(2-m)(3-m)} \frac{\d}{\d t} \left( \int_\Omega u_\vep^{3-m} \rho^2 \, \d x \right) + \frac{\lambda m}2 \|\nabla u_\vep\|_{L^2(\omega)}^2 \leq C_\omega.
\end{equation}
Here and henceforth, we shall denote by $C_\omega$ general constants which depend on $\omega \Subset \Omega$ but are independent of $\vep \in (0,1)$, $x \in \Omega$ and $t \in (0,T)$ and which may vary from line to line. In case $1 < m < 2$, one can obtain uniform estimates even for (possibly) sign-changing solutions (see (iii) of Lemma \ref{L:bdd} and~\cite[Lemma 4.1]{AO1} for details). In case $2 < m < 3$, integrating both sides of \eqref{ei1} over $(0,t)$, we find that
\begin{align*}
\lefteqn{
\frac1{(m-2)(3-m)} \int_\Omega u_0^{3-m} \rho^2\, \d x + \frac{\lambda m}2 \int^t_0 \|\nabla u_\vep(\tau)\|_{L^2(\omega)}^2 \, \d \tau
}\\
&\leq \frac1{(m-2)(3-m)} \int_\Omega u_\vep(\cdot,t)^{3-m} \rho^2\, \d x + C_\omega\\
&\leq \frac{|\Omega|^{(m-1)/2}}{(m-2)(3-m)} \left( \int_\Omega u_\vep(\cdot,t)^2\, \d x \right)^{(3-m)/2} + C_\omega.
\end{align*}
Hence without imposing any additional assumptions, we can obtain
\begin{equation}\label{ee1}
\int^T_0 \|\nabla u_\vep(\tau)\|_{L^2(\omega)}^2 \, \d \tau \leq C_\omega
\end{equation}
for any $\omega \Subset \Omega$.

Now, let us discuss how to justify the formal argument so far. For the case where $u_\vep(\cdot,t)$ lies on $H^1_0(\Omega) \cap L^\infty(\Omega)$ for a.e.~$t \in (0,T)$, we test \eqref{eq:1.1} by
$$
\sigma_{\delta}(u_\vep) \rho^2 := (2-m)^{-1} \left( u_\vep + \delta \right)^{2-m} \rho^2
$$
for any $\delta \in (0,1)$ (in what follows, $C_\omega$ will also be independent of $\delta$). Since $u_\vep \geq 0$, we can assure that $\sigma_\delta(u_\vep) \in H^1_0(\Omega)$ for any $\delta > 0$. Then we find that
\begin{align*}
\left\langle \partial_t u_\vep, \sigma_{\delta}(u_\vep) \rho^2 \right\rangle_{H^1_0(\Omega)}
&= \frac{\d}{\d t} \left( \int_\Omega \hat \sigma_{\delta} (u_\vep) \rho^2 \, \d x\right)\\
&= \frac 1 {(2-m)(3-m)} \frac{\d}{\d t} \left( \int_\Omega \left( u_\vep +\delta \right)^{3-m} \rho^2 \, \d x \right),
\end{align*}
where $\hat \sigma_\delta$ denotes a primitive function of $\sigma_\delta$, and
\begin{align*}
\lefteqn{
 \int_\Omega a_\vep \nabla u_\vep^m \cdot \nabla \sigma_\delta(u_\vep) \rho^2 \, \d x
}\\
&= m \int_\Omega u_\vep^{m-1} (u_\vep + \delta)^{1-m} a_\vep \nabla u_\vep \cdot (\nabla u_\vep) \rho^2 \, \d x\\
&\quad + \frac{2m}{2-m} \int_\Omega u_\vep^{m-1} (u_\vep + \delta)^{2-m} a_\vep \nabla u_\vep \cdot (\nabla \rho) \rho \, \d x\\
&\geq \lambda m \int_\Omega \left(\frac{u_\vep}{u_\vep + \delta}\right)^{m-1} |\nabla u_\vep|^2 \rho^2 \, \d x\\
&\quad - \frac{2mM_\rho}{m-2} \int_\Omega u_\vep \left(\frac{u_\vep}{u_\vep + \delta}\right)^{m-2} |\nabla u_\vep| \rho \, \d x\\
&\geq \frac{\lambda m}2 \int_\Omega \left(\frac{u_\vep}{u_\vep + \delta}\right)^{m-1} |\nabla u_\vep|^2 \rho^2 \, \d x - C_\omega \int_\Omega u_\vep^{m-1}(u_\vep + \delta)^{3-m} \, \d x.
\end{align*}
Here the last term of the right-hand side turns out to be uniformly bounded for $\vep \in (0,1)$ with the use of Lemma \ref{L:bdd}. Hence the integration in time implies
\begin{align*}
\frac1{(m-2)(3-m)} \int_\Omega u_0^{3-m} \rho^2\, \d x + \frac{\lambda m}2 \int^t_0 \int_\Omega \left(\frac{u_\vep}{u_\vep + \delta}\right)^{m-1} |\nabla u_\vep|^2 \rho^2 \, \d x \d \tau
\\
\leq \frac{|\Omega|^{(m-1)/2}}{(m-2)(3-m)} \left( \int_\Omega \left[u_\vep(\cdot,t)+\delta\right]^2\, \d x \right)^{(3-m)/2} + C_\omega.
\end{align*}
Letting $\delta \to 0_+$ and applying Lebesgue's dominated convergence theorem, one can obtain the same conclusion. For general cases, we first consider approximate \emph{classical} solutions $u_\vep^{(n)} = u_\vep^{(n)}(x,t) \geq 1/n$ (hence, in particular, it lies on $H^1_0(\Omega) \cap L^\infty(\Omega)$) for \eqref{eq:1.1}--\eqref{eq:1.3} with smooth bounded initial data $u_0^{(n)} \geq 1/n$ satisfying $u_0^{(n)} = 1/n$ on $\partial \Omega$ and the inhomogeneous Dirichlet condition $u_\vep^{(n)} = 1/n$ on $\partial \Omega \times (0,T)$ as in~\cite[\S 5.4]{Vazquez} and derive the uniform estimates for $u_\vep^{(n)}$ with $\delta > 0$ obtained above. Then passing to the limit as $n \to +\infty$ first and then as $\delta \to 0_+$, we can obtain the desired estimates for the limit $u_\vep$ of $(u_\vep^{(n)})$, which is the unique nonnegative weak solution to \eqref{eq:1.1}--\eqref{eq:1.3}. We shall exhibit only formal derivations for the other cases below; however, one can similarly justify them with slight modifications, which remain for the reader.

In case $m > 3$, assuming that $u_0^{3-m} \in L^1_{\rm loc}(\Omega)$, which also entails the positivity of $u_0$ over $\Omega$ (but $u_0$ may still vanish on $\partial \Omega$), we have
\begin{align*}
\frac1{(m-2)(m-3)} \int_\Omega u_\vep(\cdot,t)^{3-m} \rho^2\, \d x + \frac{\lambda m}2 \int^t_0 \|\nabla u_\vep(\tau)\|_{L^2(\omega)}^2 \, \d \tau
\\
\leq \frac1{(m-2)(m-3)} \int_\Omega u_0^{3-m} \rho^2\, \d x + C_\omega.
\end{align*}
Thus we obtain
\begin{align*}
\sup_{t \in [0,T]} \left( \int_\Omega u_\vep(\cdot,t)^{3-m} \rho^2\, \d x \right) + \int^T_0 \|\nabla u_\vep(\tau)\|_{L^2(\omega)}^2 \, \d \tau
\\
\leq C \int_\Omega u_0^{3-m} \rho^2\, \d x + C_\omega
\end{align*}
for any $\omega \Subset \Omega$.

We next handle the exceptional cases, $m=2,3$. In case $m=2$, test \eqref{eq:1.1} by $\rho^2 \log u_\vep$. Then it follows that
\begin{equation}\label{ei2}
\frac{\d}{\d t} \left( \int_\Omega [u_\vep \log u_\vep - u_\vep] \rho^2 \, \d x \right) + \lambda \|\nabla u_\vep\|_{L^2(\omega)}^2 \leq C_\omega.
\end{equation}
Here we used the fact that
\begin{align*}
\lefteqn{
 \int_\Omega a_\vep \nabla u_\vep^2 \cdot \nabla \left(\rho^2 \log u_\vep\right) \, \d x
}\\
&= 2 \int_\Omega a_\vep \nabla u_\vep \cdot (\nabla u_\vep) \rho^2 \, \d x
+ 4 \int_\Omega a_\vep \nabla u_\vep \cdot (\nabla \rho) u_\vep (\log u_\vep) \rho \, \d x\\
&\geq 2\lambda \int_\Omega |\nabla u_\vep|^2 \rho^2 \, \d x - 4 M_\rho \int_\Omega |\nabla u_\vep| \rho |u_\vep \log u_\vep| \, \d x\\
&\geq \lambda \|\nabla u_\vep\|_{L^2(\omega)}^2 - C_\omega \int_\Omega |u_\vep \log u_\vep |^2 \, \d x \geq \lambda \|\nabla u_\vep\|_{L^2(\omega)}^2 - C_\omega.
\end{align*}
Here the last inequality follows from (i) of Lemma \ref{L:bdd}. Integrating \eqref{ei2} over $(0,t)$, we see that
\begin{align*}
\int_\Omega [u_\vep(\cdot,t) \log u_\vep(\cdot,t) - u_\vep(\cdot,t)] \rho^2 \, \d x + \lambda \int^t_0 \|\nabla u_\vep(\tau)\|_{L^2(\omega)}^2 \, \d \tau\\
\leq \int_\Omega [u_0 \log u_0 - u_0] \rho^2 \, \d x + C_\omega,
\end{align*}
which implies \eqref{ee1} for any $\omega \Subset \Omega$, since $r \log r - r$ is bounded from below for $r \geq 0$ and $u_0 \log u_0 \in L^1_{\rm loc}(\Omega)$ under \eqref{H1}.

In case $m = 3$, test \eqref{eq:1.1} by $- u_\vep^{-1} \rho^2$ to see that
\begin{align*}
 \dfrac{\d}{\d t} \left(\int_\Omega [-\log u_\vep]\rho^2 \, \d x\right)
 + \int_\Omega a_\vep \nabla u_\vep^3 \cdot \nabla (- u_\vep^{-1} \rho^2) \, \d x =0.
\end{align*}
Note that
\begin{align*}
\lefteqn{
\int_\Omega a_\vep \nabla u_\vep^3 \cdot \nabla (- u_\vep^{-1} \rho^2) \, \d x
}\\
&= 3 \int_\Omega a_\vep \nabla u_\vep \cdot (\nabla u_\vep) \rho^2 \, \d x - 6 \int_\Omega u_\vep a_\vep \nabla u_\vep \cdot (\nabla \rho) \rho \, \d x\\
&\geq 3 \lambda \int_\Omega |\nabla u_\vep|^2 \rho^2 \, \d x - 6 M_\rho \int_\Omega |\nabla u_\vep| \rho |u_\vep| \, \d x\\
&\geq \lambda \int_\Omega |\nabla u_\vep|^2 \rho^2 \, \d x - C_\omega \int_\Omega |u_\vep|^2 \, \d x\\
&\geq \lambda \|\nabla u_\vep\|_{L^2(\omega)}^2 - C_\omega.
\end{align*}
Here (i) of Lemma \ref{L:bdd} was used again to derive the last inequality. Thus the integration in time over $(0,t)$ yields
\begin{align*}
\lefteqn{
 \int_{[u_\vep(\cdot,t)\leq 1]} [-\log u_\vep(\cdot,t)] \rho^2 \, \d x
+ \int_{[u_0\geq 1]} (\log u_0) \rho^2 \, \d x + \lambda \int^t_0 \|\nabla u_\vep(\tau)\|_{L^2(\omega)}^2 \, \d \tau
}\\
&\leq \int_{[u_0\leq 1]} [-\log u_0] \rho^2 \, \d x + \int_{[u_\vep(\cdot,t)\geq 1]} [\log u_\vep(\cdot,t)]\rho^2 \, \d x + C_\omega,
\end{align*}
whence follows from (i) of Lemma \ref{L:bdd} that
$$
\int_{[u_\vep(\cdot,t)\leq 1]} [-\log u_\vep(\cdot,t)] \rho^2 \, \d x + \lambda \int^t_0 \|\nabla u_\vep(\tau)\|_{L^2(\omega)}^2 \, \d \tau \leq C_\omega,
$$
provided that $\log u_0 \in L^1_{\rm loc}(\Omega)$. Hence we have the following:

\begin{lemma}[Local uniform estimates for $\nabla u_\vep$ for $m \geq 2$]\label{L:lue}
Under the same assumptions as in Theorem \ref{T:AO1}, for any $\omega \Subset \Omega$, there exists a constant $C_\omega \geq 0$ such that the following holds true\/{\rm :}
\begin{enumerate}
 \item[(i)] In case $2 \leq m < 3$,
$$
\int^T_0 \|\nabla u_\vep(t)\|_{L^2(\omega)}^2 \, \d t \leq C_\omega.
$$
 \item[(ii)] In case $m = 3$,
$$
\sup_{t \in [0,T]} \left( \int_{[u_\vep(\cdot,t)\leq 1] \cap \omega} [-\log u_\vep(\cdot,t)] \, \d x \right) + \int^T_0 \|\nabla u_\vep(t)\|_{L^2(\omega)}^2 \, \d t \leq C_\omega,
$$
provided that $\log u_0 \in L^1_{\rm loc}(\Omega)$.
\item[(iii)] In case $m > 3$,
$$
\sup_{t \in [0,T]} \left( \int_\omega u_\vep(\cdot,t)^{3-m}\, \d x \right) + \int^T_0 \|\nabla u_\vep(t)\|_{L^2(\omega)}^2 \, \d t \leq C_\omega,
$$
provided that $u_0^{3-m} \in L^1_{\rm loc}(\Omega)$.
\end{enumerate}
Moreover, there exists a {\rm (}not relabeled{\rm )} subsequence of $(\vep_n)$ such that the distributional gradient $\nabla u$ lies on $[L^2_{\rm loc}(\Omega \times [0,T])]^N$ and
\begin{align*}
\nabla u_{\vep_n} \to \nabla u \quad \mbox{ weakly in } L^2(0,T;L^2(\omega))
\quad \mbox{ as } \ \vep_n \to 0_+
\end{align*}
for any $\omega \Subset \Omega$.
\end{lemma}

\begin{proof}
It remains only to show the last assertion on the weak convergence of $\nabla u_{\vep_n}$ as $\vep_n \to 0_+$. In each case of Lemma \ref{L:lue}, for each $\omega \Subset \Omega$, from the local uniform estimate for $\nabla u_\vep$, one can take a subsequence $\vep_n \to 0_+$ and a limit $\xi_\omega \in [L^2(0,T;L^2(\omega))]^N$ such that
\begin{equation}\label{c:nabla_u:omega}
\nabla u_{\vep_n} \to \xi_\omega \quad \mbox{ weakly in } [L^2(0,T;L^2(\omega))]^N,
\end{equation}
and then, $\xi_\omega(\cdot,t)$ coincides with the distributional gradient $\nabla u(\cdot,t)$ of $u(\cdot,t)$ over $\omega$. Hence we can simply write $\nabla u$ instead of $\xi_\omega$. Set $\omega_m := \{x \in \Omega \colon \mathrm{dist}(x,\partial) > 1/m\}$ for $m \in \N$, where $\mathrm{dist}(x,\partial)$ stands for the distance between $x$ and the boundary $\partial \Omega$. Then for any $\omega \Subset \Omega$ one can take $m \in \N$ large enough such that $\omega \subset \omega_m$. Hence, due to a diagonal argument, one can take a (not relabeled) subsequence such that \eqref{c:nabla_u:omega} holds for any $\omega \Subset \Omega$; moreover, $\nabla u$ turns out to be Lebesgue measurable in $\Omega \times (0,T)$. Therefore $\nabla u$ belongs to $[L^2_{\rm loc}(\Omega \times [0,T])]^N$.
\end{proof}

We close this section with the following:

\begin{remark}[Optimal integrability of the gradient $\nabla u_\vep$]\label{R:opt}
In Lemmas \ref{L:bdd} and \ref{L:lue}, we have established global and local (in space) uniform estimates in $L^2$ for $\nabla u_\vep$ when $m < 2$ and $m \geq 2$, respectively. This remark is devoted to discussing the necessity of assumptions in those lemmas in view of optimal integrability for the gradient $\nabla u_\vep$. Let us first focus on a difference between $m < 2$ and $m \geq 2$. For simplicity, let $a(y,s)$ be just an $N \times N$ unit matrix and let $\Omega$ be the unit ball $B_1(0) \subset \R^N$ centered at the origin. Then \eqref{eq:1.1}--\eqref{eq:1.3} admits a separable solution, that is, $u(x,t) = \rho(t) \phi(x)$, where $\rho(t) > 0$ and $\phi(x) > 0$ solve
$$
\partial_t \rho = - \lambda \rho^m, \quad \rho(0)=1, \quad - \Delta \phi^m = \lambda \phi \ \mbox{ in } B_1(0), \quad \phi = 0 \ \mbox{ on } \partial B_1(0)
$$ 
for any $\lambda > 0$. Then $\phi$ is radially symmetric, that is, $\phi(x) = \tilde\phi(r)$ for $r = |x|$. Moreover, thanks to Hopf's lemma, it follows that $\partial_r (\tilde\phi^m)(1) < 0$, and hence, one may write 
$$
|\partial_r \tilde\phi(r)| = \frac1m \tilde \phi(r)^{1-m} |\partial_r (\tilde\phi^m)(r)| \asymp (1-r)^{(1-m)/m}
$$ 
near the boundary $r = 1$, that, is, $|\nabla \phi(x)| = |\partial_r {\tilde\phi}(r)| \asymp (1-r)^{(1-m)/m}$ for $r$ close to $1$ (here, $f \asymp g$ means $c f(\cdot) \leq g(\cdot) \leq c^{-1}f(\cdot)$ for some constant $c > 0$). Therefore observing that
$$
\int^1_{1-\vep} (1-r)^{2(1-m)/m} \, \d r < +\infty \quad \mbox{ if and only if } \quad m < 2
$$
for any $\vep \in (0,1)$, one can expect the global regularity $\nabla u_\vep(\cdot,t) \in L^2(\Omega)$ for $m < 2$ only, even though $\nabla u_\vep^m(\cdot,t)$ always lies on $L^2(\Omega)$.

We next discuss the necessity of the positivity of initial data in $\Omega$ to derive local $L^2$ estimates for $\nabla u_\vep(\cdot,t)$ for $m \geq 3$. To this end, recall the so-called \emph{Barenblatt solution} (or \emph{Zel'dovich-Kompaneets-Barenblatt solution}),
$$
\mathcal{B}(x,t) = t^{-\alpha} \left[ C - \kappa (t^{-\alpha/N} |x|)^2 \right]_+^{1/(m-1)} \ \mbox{ for } \ x \in \mathbb R^N, \ t > 0,
$$
where $\alpha := \frac{N}{N(m-1)+2}$ and $\kappa := \frac{\alpha(m-1)}{2Nm} > 0$, for any $C > 0$. For simplicity, set $C = \kappa$. We remark that $\mathcal{B}(x,t)$ solves \eqref{eq:1.1}--\eqref{eq:1.3}, whenever $t < 1$, for which the support of $\mathcal{B}(\cdot,t)$ is still inside $\Omega = B_1(0)$. Let us calculate the $L^2$ norm of the gradient $\nabla \mathcal{B}(\cdot,t)$ for $t \in (0,1)$. Setting $\sigma = (t^{-\alpha/N}r)^2$, we find that
\begin{align*}
\lefteqn{
\int_{B_1(0)} |\nabla \mathcal{B}(x,t)|^2 \, \d x
}\\
&= |B_1(0)| \frac{4\kappa^{2/(m-1)}}{(m-1)^2} t^{-2\alpha(1+2/N)} \int^{t^{\alpha/N}}_0 \left[ 1 - (t^{-\alpha/N}r)^2 \right]^{2/(m-1)-2} r^{N-1} \, \d r\\
&= |B_1(0)| \frac{2\kappa^{2/(m-1)}}{(m-1)^2} t^{-\alpha(1+4/N)} \int^1_0 (1 - \sigma)^{2/(m-1)-2} \sigma^{(N-2)/2} \, \d \sigma,
\end{align*}
where $|B_1(0)|$ stands for the volume of the unit ball $B_1(0)$. Thus we observe that
$$
\nabla \mathcal{B}(\cdot,t) \in L^2_{\rm loc}(B_1(0)) \quad \mbox{ if and only if } \quad m < 3.
$$
Therefore even local regularity $\nabla u_\vep(\cdot,t) \in L^2_{\rm loc}(\Omega)$ cannot be expected for $m \geq 3$ when initial data vanish in $\Omega$ similarly to Barenblatt solutions whose supports are included in domains.

Finally, we remark that even local uniform estimates obtained above cannot be extended to sign-changing weak solutions for \eqref{eq:1.1}--\eqref{eq:1.3} for $m \geq 2$ (cf.~\cite{AO1}). Indeed, a sign-changing explicit solution called a \emph{dipole solution} defined over the interval $[-1,1]$ has a compact support expanding at a finite speed and a fixed zero at the origin, and moreover, the local behavior of the dipole solution near the origin is similar to $|x|^{1/m}\mathrm{sign}(x)$ (see~\cite[p.69]{Vazquez}). Moreover, we observe that
$$
\int^\vep_0 |(x^{1/m})'|^2 \, \d x < +\infty \quad \mbox{ if and only if } \quad m < 2
$$
for any $\vep \in (0,1)$. Indeed, calculating the $L^2$ norm of the gradient of the dipole solution over any open interval $I \Subset (-1,1)$ involving the origin, one can find that it is finite only when $m < 2$. Therefore the nonnegativity of initial data, which is assumed throughout the present paper, is necessary for $m \geq 2$ to derive local uniform estimates in $L^2$ for $\nabla u_\vep(\cdot,t)$ as in Lemma \ref{L:lue}.
\end{remark}

\section{Proof of Theorem \ref{T:main}}\label{S:prf}

We first recall the following:
\begin{lemma}[\cite{AO1}]\label{L:key}
Under the same assumptions as in Theorem \ref{T:main}, it holds that
\begin{align*}
  \lefteqn{\lim_{\vep_n\to 0_+}\vep_n^{2-r}\int_0^T\int_{\Omega}\frac{u_{\vep_n}(x,t)-u(x,t)}{\vep_n}\phi(x)b(\tfrac{x}{\vep_n})\psi(t)\partial_s c(\tfrac{t}{\vep_n^r})\, \d x\d t}\nonumber\\
&= \vierint a(y,s) \left[ \nabla u^m(x,t) + \nabla_y z(x,t,y,s) \right] \cdot \phi(x) \nabla_yb(y)\psi(t)c(s)\, \d Z\nonumber
\end{align*}
for any $\phi\in C^{\infty}_{\rm c}(\Omega)$, $b\in C^{\infty}_{\mathrm{per}}(\square)$, $\psi \in C^{\infty}_{\rm c}(0,T)$ and $c\in C^{\infty}_{\mathrm{per}}([0,1])$.
\end{lemma}

This lemma has already been proved in~\cite[Lemma 5.3]{AO1}; however, for the completeness, we briefly give a proof.
\begin{proof}
Subtracting the weak form of \eqref{eq:homeq} from that of \eqref{eq:1.1} and testing both sides by $\Psi_{\vep_n}(x,t):=\vep_n\phi(x) b(\tfrac{x}{\vep_n})\psi(t)c(\tfrac{t}{\vep_n^r})$, we observe that
\begin{align*}
0&= -\int_0^T\int_{\Omega}\left[u_{\vep_n}(x,t)-u(x,t)\right]\partial_t\Psi_{\vep_n}(x,t)\, \d x\d t\\
&\quad +\int_0^T\int_{\Omega}\left[ a(\tfrac{x}{\vep_n},\tfrac{t}{\vep_n^{r}})\nabla u_{\vep_n}^m(x,t) -j_{\rm hom}(x,t)\right] \cdot
\nabla \Psi_{\vep_n}(x,t)\, \d x\d t\nonumber\\
&=
-\vep_n\int_0^T\int_{\Omega}\bigl[u_{\vep_n}(x,t)-u(x,t)\bigl]\phi(x)b(\tfrac{x}{\vep_n})\partial_t\psi(t)c(\tfrac{t}{\vep_n^r})\, \d x\d t\nonumber\\
&\quad -\vep_n^{2-r}\int_0^T\int_{\Omega}\frac{u_{\vep_n}(x,t)-u(x,t)}{\vep_n}\phi(x)b(\tfrac{x}{\vep_n})\psi(t)\partial_s c(\tfrac{t}{\vep_n^r}) \, \d x\d t\nonumber\\
&\quad +\vep_n\int_0^T\int_{\Omega}\left[ a(\tfrac{x}{\vep_n},\tfrac{t}{\vep_n^{r}})\nabla u_{\vep_n}^m(x,t) -j_{\rm hom}(x,t)\right] \\
&\qquad \cdot \, \nabla\phi(x)b(\tfrac{x}{\vep_n})\psi(t)c(\tfrac{t}{\vep_n^r})\, \d   x\d t\nonumber\\
&\quad +\int_0^T\int_{\Omega}a(\tfrac{x}{\vep_n},\tfrac{t}{\vep_n^{r}})\nabla u_{\vep_n}^m(x,t) \cdot \phi(x)\nabla_yb(\tfrac{x}{\vep_n})\psi(t)c(\tfrac{t}{\vep_n^r})\, \d x\d t\nonumber\\
&\quad -\int_0^T\int_{\Omega}j_{\rm hom}(x,t)\cdot \phi(x)\nabla_yb(\tfrac{x}{\vep_n})\psi(t)c(\tfrac{t}{\vep_n^r})\, \d x\d t. \nonumber
\end{align*}
Thanks to Theorem \ref{T:AO1}, the first and third terms vanish as $\vep_n\to 0_+$. Furthermore, the mean value property (see \cite[Theorem 2.6]{CD}) and the periodicity of $b\in C^{\infty}_{\rm per}(\square)$ yield that the fifth term also vanishes as $\vep_n\to 0_+$. Making use of \eqref{eq:Thm1.1-3}, we get the assertion.
\end{proof}


In order to prove Theorem \ref{T:main} for $r\ge 2$, based on Lemma \ref{L:lue}, we need the following:

\begin{lemma}[Relation between correctors]\label{L:zw}
Under the same assumptions as in Theorem \ref{T:main}, let $u_{\vep_n}$ be the unique weak solution on $[0,T]$ to \eqref{eq:1.1}--\eqref{eq:1.3} and let $u$ be the limit of $(u_{\vep_n})$ as in Theorem \ref{T:AO1} as $\vep_n\to 0_+$. In addition, let $z\in L^2(\Omega\times (0,T);L^2(0,1;H^1_{\mathrm{per}}(\square)/\R))$ be a function obtained in Theorem \ref{T:AO1}. Then there exist a {\rm (}not relabeled{\rm )} subsequence of $(\vep_n)$ and $w\in L^2_{\rm loc}(\Omega \times [0,T] ;L^2(0,1;H^1_{\mathrm{per}}(\square)/\R))$ such that
\begin{equation}\label{eq:w}
\nabla u_{\vep_n} \wtts \nabla u + \nabla_y w\ \text{ in }
[L^2(\omega \times (0,T) \times \square \times (0,1))]^N
\end{equation}
for any $\omega\Subset \Omega$, and moreover, it holds that
\begin{equation}\label{z-w}
z(x,t,y,s) =mu^{m-1}(x,t)w(x,t,y,s)
\end{equation}
for a.e.~$(x,t,y,s)\in \Omega\times (0,T)\times \square\times (0,1)$.
\end{lemma}

\begin{proof}
By Theorem \ref{T:wttscpt_grad} and Lemma \ref{L:lue} along with \eqref{eq:Thm1.1-2} and \eqref{hyp-u0}, for each $\omega \Subset \Omega$, one can take a (not relabeled) subsequence of $(\vep_n)$ and a function $w_\omega$, which may depend on the choice of $\omega$, such that
$$
\nabla u_{\vep_n} \wtts \nabla u + \nabla_y w_\omega \ \mbox{ in } [L^2(\omega \times (0,T) \times \square \times (0,1))]^N.
$$
Here we recall that the weak two-scale limit $\nabla u$ coincides with the distributional gradient of $u$, and hence, it is independent of the choice of $\omega$ (see Lemma \ref{L:lue}). Taking an increasing sequence of domains $\omega_m := \{x \in \Omega \colon \mathrm{dist}(x,\partial \Omega) > 1/m\} \Subset \Omega$, by a diagonal argument, we can construct a (not relabeled) subsequence of $(\vep_n)$ and a Lebesgue measurable function $w : \Omega \times (0,T) \times \square \times (0,1) \to \R$ such that \eqref{eq:w} holds true. Indeed, we find that $w_{\omega_m}(\cdot,t,y,s) = w_{\omega_{m'}}(\cdot,t,y,s)$ in $\omega_m$ if $m \leq m'$. Hence for each $x \in \Omega_0 := \cup_{m\in \N} \omega_m$, we define $w(x,t,y,s) = w_{\omega_m}(x,t,y,s)$, where $m$ is an integer such that $x \in \omega_m$. Since $|\Omega \setminus \Omega_0| = 0$, we obtain \eqref{eq:w}. Moreover, for any $\omega \Subset \Omega$, one can take $m \in \N$ such that $\omega \subset \omega_m$, and therefore, we conclude that $w \in L^2_{\rm loc}(\Omega \times [0,T]; L^2(0,1;H^1_{\rm per}(\square) / \R))$.

We next prove \eqref{z-w}. Furthermore, recalling \eqref{eq:Thm1.1-2} in Theorem \ref{T:AO1}, that is,
\begin{equation}\label{cs}
u_{\vep_n}^m\to u^m\quad \text{ strongly in }\ L^{\rho}(0,T;L^{(m+1)/m}(\Omega))
\end{equation}
for any $\rho\in (1,+\infty)$, we can verify that
\begin{align*}
\lefteqn{
\left\|u_{\vep_n}^{m-1}-u^{m-1}\right\|_{L^{(m+1)/(m-1)}(\Omega\times (0,T))}^{(m+1)/(m-1)}
}\\
 &\le \int_0^T\int_{\Omega}\left|u_{\vep_n}^m(x,t)-u^m(x,t)\right|^{(m+1)/m}\, \d x\d t \to 0\quad\text{ as } \ \vep_n\to 0_+.
\end{align*}
Hence in case $m \leq 3$, it follows immediately that
\begin{equation}\label{eq:L2s}
u_{\vep_n}^{m-1}\to u^{m-1}\quad \text{ strongly in }\ L^{2}(\Omega\times (0,T)).
\end{equation}
In case $m > 3$, recall that $(u_\vep^m)$ is bounded in $L^2(0,T;H^1_0(\Omega))$
(see Lemma \ref{L:bdd}). Since $1 < (m+1)/(m-1) < 2 < 2m/(m-1)$, there exists $0<\theta<1$ such that
\begin{align*}
\lefteqn{
\left\|u_{\vep_n}^{m-1}-u^{m-1}\right\|_{L^{2}(\Omega\times (0,T))}
}
\\
&\leq \left\|u_{\vep_n}^{m-1}-u^{m-1}\right\|_{L^{(m+1)/(m-1)}(\Omega\times (0,T))}^{1-\theta}
\left\|u_{\vep_n}^{m-1}-u^{m-1}\right\|_{L^{2m/(m-1)}(\Omega\times (0,T))}^{\theta}\\
&\leq C\left\|u_{\vep_n}^{m-1}-u^{m-1}\right\|_{L^{(m+1)/(m-1)}(\Omega\times (0,T))}^{1-\theta}\to 0\quad \text{ as } \ \vep_n\to 0_+.
\end{align*}
Here we have used the fact (by Poincar\'e's inequality) that
\begin{align*}
\left\|u_{\vep_n}^{m-1}-u^{m-1}\right\|_{L^{2m/(m-1)}(\Omega\times (0,T))}^{2m/(m-1)}
\leq \left\|u_{\vep_n}^m-u^m\right\|_{L^2(\Omega\times (0,T))}^2 \leq C. 
\end{align*}
Hence \eqref{eq:L2s} holds for $m>3$ as well. Moreover, one can verify from \eqref{eq:w} that
\begin{equation}\label{BBB}
\nabla u_{\vep_n} \cdot B(\tfrac{x}{\vep_n}) c(\tfrac{t}{\vep_n^r}) \to \int^1_0 \int_{\square} \left( \nabla u + \nabla_y w \right) \cdot B(y) c(s) \, \d y \d s
\end{equation}
weakly in $[L^2(\omega \times (0,T))]^N$ for any $\omega \Subset \Omega$, $B \in [C^\infty_{\rm per}(\square)]^N$ and $c \in C^\infty_{\rm per}([0,1])$. For all $\phi\in C^{\infty}_{\rm c}(\Omega)$ and $\psi\in C^{\infty}_{\rm c}(0,T)$ along with $B$ and $c$ as above, one can derive from \eqref{eq:L2s} and \eqref{BBB} that
\begin{align*}
\lefteqn{\int_0^T\int_{\Omega}\nabla u_{\vep_n}^m(x,t) \cdot \phi(x)B(\tfrac{x}{\vep_n})\psi(t)c(\tfrac{t}{\vep_n^r}) \, \d x\d t}
\\
&= \int_0^T\int_{\Omega} mu_{\vep_n}^{m-1}(x,t)\nabla u_{ \vep_n  }(x,t)\cdot  \phi(x)B(\tfrac{x}{\vep_n})\psi(t)c(\tfrac{t}{\vep_n^r})\, \d x\d t\nonumber\\
&= \vierint mu^{m-1}(x,t) \left[ \nabla u(x,t)+\nabla_y w(x,t,y,s) \right]\\
&\qquad \cdot \, \phi(x)B(y)\psi(t) c(s)\, \d Z,
\end{align*}
since the support of $\phi$ is compact in $\Omega$, i.e., $\omega := \mathrm{supp} \, \phi \Subset \Omega$. Moreover, recalling \eqref{eq:Thm1.1-3}, we conclude that
\begin{align*}
\vierint \nabla_y \left[ mu^{m-1}(x,t)w(x,t,y,s)-z(x,t,y,s)\right]\\
\cdot\, \phi(x) B(y)\psi(t)c(s)\, \d Z =0,
\end{align*}
which together with the arbitrariness of test functions
and the zero mean in $\square$ of $z$ and $w$ yield the assertion.
\end{proof}

We shall prove Theorem \ref{T:main} by applying Lemma \ref{L:key} for all $\phi\in C^{\infty}_{\rm c}(\Omega)$, $b\in C^{\infty}_{\mathrm{per}}(\square)/\R$, $\psi \in C^{\infty}_{\rm c}(0,T)$ and $c\in C^{\infty}_{\mathrm{per}}([0,1])$.

\begin{proof}[Proof of Theorem \ref{T:main}]
Thanks to Lemma \ref{L:lue} along with \eqref{cs}, by \eqref{eq:w} of Lemma \ref{L:zw} and Corollary \ref{C:vw}, we obtain
\begin{align}
\lim_{\vep_n\to 0_+} &\int_{0}^{T}\int_{\Omega}\frac{u_{\vep_n}(x,t)-u(x,t)}{\vep_n} \phi(x)b(\tfrac{x}{\vep_n})\psi(t)\partial_sc(\tfrac{t}{\vep_n^{r}})\, \d x\d t\label{sharp}\\
&= \vierint w(x,t,y,s) \phi(x)b(y)\psi(t)\partial_sc(s)\, \d Z\nonumber
\end{align}
for any $\phi\in C^{\infty}_{\rm c}(\Omega)$, $b\in C^{\infty}_{\mathrm{per}}(\square)/\R$ {\rm (}i.e., $\langle b \rangle_y = 0${\rm )}, $\psi\in C^{\infty}_{\rm c}(0,T)$ and $c\in C^{\infty}_{\mathrm{per}}([0,1])$, since the support of $\phi$ is compact in $\Omega$. 

For the case where $0<r<2$, Lemma \ref{L:key} along with \eqref{sharp} yields that
\begin{equation}\label{eq:cpz}
\int_0^1\int_\square a(y,s)\bigl[ \nabla u^m(x,t)+\nabla_y z(x,t,y,s)\bigl]\cdot \nabla_yb(y)c(s)\, \d y\d s=0
\end{equation}
for a.e.~$(x,t) \in \Omega\times (0,T)$; here we have used the arbitrariness of $\phi\in C^{\infty}_{\rm c}(\Omega)$ and $\psi\in C^{\infty}_{\rm c}(0,T)$.

Now, one can write
\begin{equation}\label{eq:zform}
z(x,t,y,s)=\sum_{k=1}^N \partial_{x_k}u^m(x,t) \Phi_k(y,s),
\end{equation}
where $\Phi_k = \Phi_k(y,s)$ is the unique weak solution to \eqref{eq:CP1}. Indeed, setting $\tilde{z}(x,t,y,s) = $ (the right-hand side of \eqref{eq:zform}), we see that \eqref{eq:cpz} with $z$ replaced by $\tilde{z}$ holds. Hence, as in~\cite{AO1}, putting $bc=(z-\tilde{z})(x,t,\cdot,\cdot)$ (via suitable approximation) and subtracting \eqref{eq:cpz} for $z$ and $\tilde{z}$, one can derive from \eqref{eq:CP1} and the Poincar\'e-Wirtinger inequality that
\begin{align}
0
&\stackrel{\eqref{eq:CP1}}=
\int_0^1\int_{\square}a(y,s)\nabla_y(z-\tilde{z})(x,t,y,s)\cdot \nabla_y(z-\tilde{z})(x,t,y,s)\, \d y\d s \label{eq:chkz}\\
&\ge
{\lambda}\|\nabla_y(z-\tilde{z})(x,t)\|^2_{L^2(\square\times (0,1))}
\ge
C_{\lambda,\square}\|(z-\tilde{z})(x,t)\|^2_{L^2(\square\times (0,1))},
\nonumber
\end{align}
where $C_{\lambda,\square} > 0$ is a constant from the Poincar\'e-Wirtinger inequality. Hence \eqref{eq:zform} follows. Thus we conclude that
\begin{align*}
j_{\rm hom}(x,t)
&=
\int_0^1\int_{\square}a(y,s)\Bigl[\nabla u^m(x,t)+\sum_{k=1}^N \partial_{x_k}u^m(x,t) \nabla_y\Phi_k(y,s)\Bigl]\, \d y\d s\\
&=
\sum_{k=1}^N\underbrace{\Bigl(\int_0^1\int_{\square}a(y,s)\left[\nabla_y\Phi_k(y,s)+e_k\right]\, \d y\d s\Bigl)}_{= \,a_{\rm hom}e_k \text{ by \eqref{eq:ahom}}}\partial_{x_k}u^m(x,t)\\
&=a_{\rm hom}\nabla u^m(x,t).
\end{align*}

For the case where $r=2$, employing \eqref{sharp} again and \eqref{z-w} of Lemma \ref{L:zw}, we find from Lemma \ref{L:key} that
\begin{align*}
0&= - \vierint w(x,t,y,s) \phi(x)b(y)\psi(t)\partial_s c(s)\, \d Z
\\
&+ \vierint a(y,s) \left[ \nabla u^m(x,t)+mu^{m-1}(x,t)\nabla_y w(x,t,y,s)\right]\\
&\qquad \cdot \phi(x) \nabla_y b(y) \psi(t) c(s)\, \d Z
\end{align*}
for any $\phi\in C^{\infty}_{\rm c}(\Omega)$, $b\in C^{\infty}_{\mathrm{per}}(\square)/\R$, $\psi\in C^{\infty}_{\rm c}(0,T)$ and $c\in C^{\infty}_{\mathrm{per}}([0,1])$. Here we recall the fact that $w$ is differentiable at $s\in (0,1)$ in the Sobolev-Bochner sense (see the proof of Theorem 1.9 for $r=2$ in~\cite{AO1}). Hence from the arbitrariness of $\phi\in C^{\infty}_{\rm c}(\Omega)$ and $\psi\in C^{\infty}_{\rm c}(0,T)$, we obtain
\begin{align*}
0&= \int^1_0 \left\langle \partial_s w(x,t,\cdot,s), b \right\rangle_{H^1_{\rm per}(\square) / \R} c(s) \, \d s
\\
&+ \int^1_0 \int_\square mu^{m-1}(x,t) a(y,s) \left[ \nabla u(x,t) + \nabla_y w(x,t,y,s)\right]\cdot \nabla_y b(y) c(s)\, \d y\d s
\end{align*}
for a.e.~$(x,t) \in \Omega \times (0,T)$. When $u(x,t)=0$, we can obtain $w(x,t,\cdot,\cdot) \equiv 0$ (and hence, $z(x,t,\cdot,\cdot) \equiv 0$). So we shall assume $u(x,t) \neq 0$ below. As in \eqref{eq:chkz}, by virtue of the $1$-periodicity of $w$ in $(y,s)$, one can verify that
\begin{equation*}
w(x,t,y,s)=\sum_{k=1}^N \left[\partial_{x_k}u^m(x,t)\right]\Psi_k(x,t,y,s),
\end{equation*}
where $\Psi_k = \Psi_k(x,t,y,s) \in L^\infty(\Omega \times (0,T) ; L^2(0,1;H^1_{\rm per}(\square)/\R))$ is the unique weak solution of \eqref{eq:CP2} (see~\cite[Proof of Theorem 1.4 and Appendix]{AO1}), and moreover, it follows from \eqref{z-w} that
\begin{equation*}
z(x,t,y,s)=\sum_{k=1}^N \left[\partial_{x_k}u^m(x,t)\right] \Phi_k(x,t,y,s)
\end{equation*}
with $\Phi_k(x,t,y,s) := mu^{m-1}(x,t) \Psi_k(x,t,y,s)$. The rest of the proof runs as in the case of $0<r<2$.

For the case where $2<r<+\infty$, note that $z=z(x,t,y)$ is independent of $s$. Indeed, let $\phi\in C^{\infty}_{\rm c}(\Omega)$, $b\in C^{\infty}_{\mathrm{per}}(\square)/\R$, $\psi\in C^{\infty}_{\rm c}(0,T)$ and $c\in C^{\infty}_{\mathrm{per}}([0,1])$. Setting $\Psi_{\vep_n}(x,t) = \vep_n^{r-1} \phi(x)b(\frac{x}{\vep_n})\psi(t)c(\tfrac{t}{\vep_n^r})$ in the proof of Lemma \ref{L:key}, we can derive
\begin{align*}
0&=
\lim_{\vep_n\to 0_+}\int_0^T\int_{\Omega}
\frac{u_{\vep_n}(x,t)-u(x,t)}{\vep_n}\phi(x)b(\tfrac{x}{\vep_n})\psi(t)\partial_s c(\tfrac{t}{\vep_n^r})\, \d x\d t\\
&=
\int_0^T\int_{\square}\int_{\Omega}\left[\int_0^1w(x,t,y,s)\partial_sc(s)\, \d s\right]
\phi(x)b(y)\psi(t)\, \d x\d y\d t.
\end{align*}
Here we used \eqref{sharp} again to derive the second equality. Hence from the arbitrariness of test functions and the zero mean of $w$ in $y$, we derive that
$$
\partial_s w(x,t,y, \cdot)=0 \quad \text{for a.e.~$(x,t,y) \in \Omega \times (0,T) \times \square$}
$$
in the distributional sense. Thus $w(x,t,y,s)$ is independent of $s\in (0,1)$, and so is $z(x,t,y,s)$ by Lemma \ref{L:zw}.

Now, setting $c(s)\equiv 1$ in Lemma \ref{L:key}, we see from the arbitrariness of $\phi\in C^\infty_{\rm c}(\Omega)$ and $\psi\in C^\infty_{\rm c}(0,T)$ that
\begin{equation*}
 \int_{\square}\Big(\int_0^1a(y,s)\, \d s\Big) \left[\nabla u^m(x,t)+\nabla_y z(x,t,y)\right]\cdot\nabla_yb(y)\, \d y=0
\end{equation*}
for a.e.~$(x,t) \in \Omega\times (0,T)$. Furthermore, recalling $z=z(x,t,y)$, as in \eqref{eq:chkz}, one can prove \eqref{eq:zform}, where $\Phi_k$ is the unique weak solution to \eqref{eq:CP3}. Hence the rest of the proof runs as in the case of $0<r<2$. This completes the proof.
\end{proof}

\end{document}